\documentclass[leqno,12pt]{amsart}
\usepackage{amsthm}
\usepackage{amssymb}
\usepackage{verbatim}

\numberwithin{equation}{section}

\theoremstyle{plain}
\newtheorem{thm}[equation]{Theorem}

\newtheorem{prop}[equation]{Proposition}
\newtheorem{cor}[equation]{Corollary}
\newtheorem{lemma}[equation]{Lemma}

\theoremstyle{definition}
\newtheorem{definition}[equation]{Definition}

\newtheorem{remark}[equation]{Remark}

\newcommand{\Z}{\mathbb Z}
\newcommand{\F}{\mathbb F}

\DeclareMathOperator{\AGL}{AGL}
\DeclareMathOperator{\PSL}{PSL}
\DeclareMathOperator{\PGL}{PGL}
\DeclareMathOperator{\PGammaL}{P{\Gamma}L}
\DeclareMathOperator{\lcm}{lcm}
\DeclareMathOperator{\Aut}{Aut}

\title{Permutation groups generated by binomials}

\author{Michael E. Zieve}
\address{
  Department of Mathematics,
  University of Michigan,
  Ann Arbor, MI 48109--1043,
  USA
}
\address{Mathematical Sciences Center, Tsinghua University, Beijing 100084, China}
\email{zieve@umich.edu}
\urladdr{www.math.lsa.umich.edu/$\sim$zieve/}

\begin{document}

\begin{abstract}
Let $G(q)$ be the group of permutations of $\F_q^*$ generated by those permutations which can
be represented as
$c\mapsto ac^m+bc^n$ with $a,b\in\F_q^*$ and $0<m<n<q$.
We show that there are infinitely many $q$ for which $G(q)$ is the group
of all permutations of $\F_q^*$.  This resolves a conjecture of Vasilyev
and Rybalkin.
\end{abstract}

\thanks{The author thanks Cheryl Praeger and Pablo Spiga for helpful correspondence about group theory, and especially
for pointing him to \cite{PS}.  The author also thanks Igor Shparlinski and Kannan Soundararajan
for discussions about Remark~\ref{sound},
and the NSF for support under grant DMS-1162181.}

\maketitle


\section{Introduction}

Let $\F_q$ be the finite field of cardinality $q$.  We will be interested in the group of permutations
of $\F_q$ induced by certain special permutations, especially by permutations of the form $c\mapsto f(c)$ where $f(x)\in\F_q[x]$ is a polynomial having a particularly simple form.  
One result along these lines is due
to Carlitz \cite{Carlitz} (see also \cite{Z0}), who showed that if $q>2$ then the symmetric group $S_q$ is generated by the permutations of $\F_q$ induced by $x^{q-2}$ and by degree-one polynomials in $\F_q[x]$.  Recently Vasilyev and Rybalkin \cite{VR} investigated the group of permutations of $\F_q$ generated by those permutations which are induced by binomials.  In order to obtain a  problem which is not solved by Carlitz's result, they required that the binomials be ``honest" binomials, in the sense that they are not monomials in disguise.  One way to disguise a monomial is to add to it some multiple of $x^q-x$, since this does not affect the function it induces on $\F_q$.  Another way to disguise a monomial is to add a constant to it, which does not affect whether or not the function induces a permutation of $\F_q$.  In light of these two operations, Vasilyev and Rybalkin restrict to binomials of the form $ax^m+bx^n$ with $a,b\in\F_q^*$ and $0<m<n<q$.  Note that any such binomial fixes 0, so if it induces a permutation of $\F_q$ then it also induces a permutation of $\F_q^*$.  We write $G(q)$ for the subgroup of $S_{q-1}$ generated by all such binomial permutations:

\begin{definition} Let $G(q)$ be the group of permutations of\/ $\F_q^*$ generated by the permutations of\/ $\F_q^*$ which can be represented as $ax^m+bx^n$ with $a,b\in\F_q^*$ and $0<m<n<q$.
\end{definition}

Vasilyev and Rybalkin conjectured \cite[Conj.~2]{VR} that there are infinitely many prime powers $q$ for which $G(q)$ equals the group $S_{q-1}$ of all permutations of $\F_q^*$.  We remark that such prime powers $q$ seem to be rare: for instance, Vasilyev and Rybalkin checked that there are precisely $18$ such $q$ with $q<5000$.  Our main result asserts that there do indeed exist infinitely many such $q$:

\begin{thm} \label{main}
There are infinitely many primes $p$ for which $G(p^2)$ equals $S_{p^2-1}$.  In fact,
\[
\liminf_{N\to\infty} \frac{\#\{\text{primes } p\le N:\, G(p^2)=S_{p^2-1}\}}{\#\{\text{primes } p\le N\}} \ge \frac{1}{96}.
\]
\end{thm}

We did not attempt to optimize the bound $1/96$ in this result.  This bound can be improved by using further arguments
of a similar nature to the arguments in our proof.  However, it is not clear to us whether this bound
can be improved to $1$.

Our proof of Theorem~\ref{main} relies on the classification of primitive subgroups of $S_n$ which contain an $n$-cycle; here a subgroup $H$ of $S_n$ is \textit{primitive} if the only partitions of $\{1,2,\dots,n\}$ which are preserved by $H$ are the trivial partitions $\{\{1,2,\dots,n\}\}$ and $\{\{1\},\{2\},\dots,\{n\}\}$.  This classification is a consequence of the classification of finite simple groups.  An unusual feature of our situation is that it is easy to construct a $(p^2-1)$-cycle in $G(p^2)$, while the difficult part of our proof is showing that $G(p^2)$ is primitive.  Fortunately, not many partitions of $\{1,2,\dots,p^2-1\}$ are preserved by
the $(p^2-1)$-cycle, and we show that any such partition besides the two trivial ones will not be preserved by some member of one of three known families of permutation binomials on $\F_{p^2}$, at least if $p$ is a sufficiently large prime which is congruent to either $5$ or $11$ mod $24$.

We do not know whether there are infinitely many primes $p$ for which $G(p)$ equals $S_{p-1}$.  In fact we do not even have a guess what the answer should be.  As we will explain, existing conjectures about permutation binomials ``almost" imply that there are only finitely many such $p$, but only if we change a certain constant in those conjectures in a way that makes them false.  It would be interesting to analyze this question further.

Related permutation groups have been considered in the literature, for instance
see \cite{Stafford}.  Most notably,
Wan and Lidl \cite{WL} determined the group $W(q,d)$ of permutations of $\F_q$ generated by all permutations induced by polynomials of the form $x^r h(x^d)$ where $d$ is a fixed divisor of $q-1$, and $r\in\Z$ and $h\in\F_q[x]$ are allowed to vary.  The Wan--Lidl group lies ``behind the scenes" for the work of the present paper, since if $d$ denotes the greatest common divisor of all the integers $\gcd(n-m,q-1)$ where there is a permutation of $\F_q$ induced by a binomial having terms of degrees $m$ and $n$ (with $0<m<n<q$), then $G(q)$ is contained in $W(q,d)$.  Thus, in order that $G(q)$ should equal $S_{q-1}$, it is necessary (but not always sufficient) that $d=1$.

This paper is organized as follows.  In the next section we review the group-theoretic results we will use.
In section~3 we present the permutation binomials which will be used in our proof.
We prove Theorem~\ref{main} in section~5, after showing that $G(p^2)$ is primitive
 for certain classes of primes $p$ in section~4.  In section 6 we discuss whether there are infinitely many primes $p$ for which $G(p)$ equals $S_{p-1}$.  We conclude in section 7 by mentioning some questions for further study.


\section{Primitive subgroups of $S_n$ containing an $n$-cycle}

In this section we recall the group-theoretic result needed in our proof.

\begin{definition} If $G$ is a subgroup of $S_n$, then a partition $\mathcal{P}$ of
$\{1,2,\dots,n\}$ is called $G$-\emph{invariant} if, for every part $S$ in $\mathcal{P}$ and every $g\in G$, the
set $g(S)$ is also a part in $\mathcal{P}$.
\end{definition}

\begin{definition}
A subgroup $G$ of $S_n$ is called \emph{primitive} if the only $G$-invariant partitions 
of 
$\{1,2,\dots,n\}$ are the trivial coarse partition $\{\{1,2,\dots,n\}\}$
and the trivial fine partition $\{\{1\}, \{2\}, \dots, \{n\}\}$.
\end{definition}

We will use the following result (whose proof relies on the classification of finite simple groups):

\begin{thm} \label{gp}
A primitive subgroups $G$ of $S_n$ contains an $n$-cycle if and only if one of the following
holds:
\begin{enumerate}
\item $G=S_n$ for some $n\ge 1$ or $G=A_n$ for some odd $n\ge 3$
\item $C_p\le G\le\AGL_1(p)$ where $n=p$ is prime
\item $\PGL_d(\ell)\le G\le\PGammaL_d(\ell)$ where $\ell$ is a prime power, $d\ge 2$, and $n=(\ell^d-1)/(\ell-1)$
\item $G=\PSL_2(11)$ or $M_{11}$, where in both cases $n=11$
\item $G=M_{23}$ where $n=23$.
\end{enumerate}
\end{thm}

In fact we will only need the following numerical consequence of Theorem~\ref{gp}:

\begin{cor} \label{gpcor}
If $n$ is even and $n\ne (\ell^d-1)/(\ell-1)$ for all integers $d\ge 2$ and prime powers $\ell$,
then the only primitive subgroup of $S_n$ which contains an $n$-cycle is $S_n$ itself.
\end{cor}

Proofs of Theorem~\ref{gp}, assuming certain previous results, are given in \cite[Thm.~3]{Jones} and \cite{McSorley}.  The
proof in \cite{Jones} relies on \cite[Thm.~4.1]{Feit}, for which the only proof in the literature is given
in \cite{McSorley}.  The proof of Theorem~\ref{gp} given in \cite{McSorley} relies on the correctness of the list in \cite[p.~8]{Cameron} of
the simple groups which occur as minimal normal subgroups of a doubly transitive group, although
the proof in \cite{Cameron} does not address the sporadic groups except by saying they ``can be handled
by \emph{ad hoc} arguments".  A detailed treatment of the sporadic groups is given in \cite{PS} (see especially Table~5.1), which verifies the claim in \cite{Cameron}, and combined with \cite{McSorley}
yields a proof of Theorem~\ref{gp}.  We remark that, besides the classification of finite simple groups,
the main work in this proof is carried out in \cite{Atlas,CKS,LPS}.

\section{Some permutation binomials}

In this section we exhibit the classes of permutation binomials which will be used in this paper.  The first class of permutation binomials appeared
in early work of Betti \cite[p.~74]{Betti} and Mathieu \cite[p.~275]{Mathieu}.

\begin{prop} \label{prop1} If $r$ is a prime power and $a\in\F_{r^k}^*$ is an element such that
$a^{(r^k-1)/(r-1)}\ne 1$,
then $f(x):=x^r-ax$ permutes\/ $\F_{r^k}$.
\end{prop}

\begin{proof} The function $c\mapsto f(c)$ induces a homomorphism from the
additive group of $\F_{r^k}$ to itself, so it is bijective if and only if
its kernel is trivial.  The kernel is trivial if and only if $a$ is not an $(r-1)$-th
power in $\F_{r^k}^*$, or equivalently $a^{(r^k-1)/(r-1)}\ne 1$.
\end{proof}

The second class of permutation binomials we will need arose in my work with Tucker \cite{TZ}.  Since that paper has not been published, I include the proof of the needed result for the reader's convenience.  I gave a slightly different proof in the recent paper \cite{Zcomp}.

\begin{prop} \label{prop2} If $r$ is a prime power with $r\equiv 2\pmod{3}$, and $a\in\F_{r^2}^*$
is such that $a^{r-1}$ has order $6/\gcd(r,2)$ in\/ $\F_{r^2}^*$, then $f(x):=x^{r+2}+ax$
permutes\/ $\F_{r^2}$.
\end{prop}

\begin{proof} If $c\in\F_{r^2}^*$ satisfies $c^{r+1}=1$, then $f(cx)=c\cdot f(x)$.  Thus, $f(\F_{r^2})$ consists of the set of $(r+1)$-th roots of the elements of $f(\F_{r^2})^{r+1}$.  The Proposition asserts that
$f(\F_{r^2})^{r+1}$ equals $\F_{r^2}^{r+1}$, or in other words equals $\F_r$.
For $c\in\F_{r^2}$ we compute
\[
f(c)^{r+1} = (c^{r+2}+ac)^{r+1} = c^{r+1} (c^{r+1}+a)^{r+1}.
\]
Writing $b:=c^{r+1}$, so that $b\in\F_r$, we have
\begin{align*}
f(c)^{r+1} &= b(b+a)^{r+1} \\
&= b(b+a)^r(b+a) \\
&=b(b+a^r)(b+a) \\
&=b^3+b^2(a+a^r)+ba^{r+1}.
\end{align*}
Since $3\nmid r$, it follows that
\[
f(c)^{r+1} = \Bigl(b+\frac{a+a^r}3\Bigr)^3 - \Bigl(\frac{a+a^r}3\Bigr)^3 + b\cdot\Bigl(a^{r+1}-\frac{(a+a^r)^2}3\Bigr).
\]
Next we compute
\[
a^{r+1} - \frac{(a+a^r)^2}3 = -\frac13(a^2-a^{r+1}+a^{2r})=-\frac{a^2}3(1-a^{r-1}+a^{2r-2}),
\]
which equals $0$ because $a^{r-1}$ is a primitive $6/\gcd(r,2)$-th root of unity and hence
is a root of the $6/\gcd(r,2)$-th cyclotomic polynomial.  Thus, for $c\in\F_{r^2}$ we have
\[
f(c)^{r+1} = \Bigl(c^{r+1}+\frac{a+a^r}3\Bigr)^3 - \Bigl(\frac{a+a^r}3\Bigr)^3.
\]
Since $r\equiv 2\pmod{3}$, we know that $x^3$ permutes $\F_r$ (because it induces a homomorphism from $\F_r^*$ to itself with trivial kernel).  It follows that $g(x):=(x+d)^3-d^3$ permutes $\F_r$, where $d:=(a+a^r)/3$.  Finally, since
$\F_{r^2}^{r+1}=\F_r$, we see that $c^{r+1}$ takes on all values in $\F_r$ when $c$ varies over $\F_{r^2}$,
so that $f(c)^{r+1}=g(c^{r+1})$ also takes on all values in $\F_r$, whence $f(\F_{r^2})^{r+1}=\F_r$, as desired.
\end{proof}

Many variants of the above permutation polynomials can be obtained using related ideas; see \cite{TZ,Z1,Z2,Z3,Z4,Zcomp} for details.

The next result comes from my joint work with Masuda, and is a part of \cite[Thm.~1.5]{MZ}:

\begin{prop} \label{prop3} Let $q\ge 4$ be a prime power, and let $0<m<n$ be integers 
such that $\gcd(m,n,q-1)=1$.  Let $T$ denote the number of values $a\in\F_q$
for which $ax^m+x^n$ permutes\/ $\F_q$, and write $s:=(q-1)/\gcd(n-m,q-1)$.
Then
\[
\frac{T}{(s-1)!} \ge \frac{q-2\sqrt{q}+1}{s^{s-1}} - (s-3)\sqrt{q} - 2.
\]
\end{prop}

We will use the following consequence of this result.

\begin{cor} \label{prop3cor} For any positive integer $s$ and any prime power $q$ with $q\equiv 1\pmod{s}$,
let $N$ denote the number of values $a\in\F_q^*$ for which $x(a+x^{(q-1)/s})$ permutes\/ $\F_q$.
If $s=2$ and $q\ge 7$ then $N>0$.
If $s=8$ and $q\ge 109951213112009$ then $N\ge 3$.
\end{cor}

\begin{proof}
If $s=2$ then Proposition~\ref{prop3} gives $T\ge (q-2\sqrt{q}+1)/2 + \sqrt{q} - 2 = (q-3)/2$, so that if $q>5$
then $N\ge T-1>0$.
If $s=8$ then Proposition~\ref{prop3} gives $T/7! \ge (\sqrt{q}-1)^2/8^7 - 5\sqrt{q} - 2$, and one can check
that this implies $T>3$ (and hence $N\ge 3$) when $q\ge 109951213112009$.
\end{proof}

\begin{remark} The value 109951213112009 in this result can be improved by various methods, but we do not see how
to improve it to a reasonably small value.
\end{remark}


\section{Primitivity of $G(q)$}

In this section we show that the group $G(q)$ is a primitive
subgroup of $S_{q-1}$ for all $q$ satisfying certain properties.  We first show that if $G(q)$ is nontrivial
then it contains a $(q-1)$-cycle.

\begin{lemma} \label{cyclelemma0} Let $q$ be a prime power for which there exist $a,b\in\F_q^*$ and $0<m<n<q$
such that $f(x):=ax^m+bx^n$ permutes\/ $\F_q$.  Then, for any generator $w$ of\/ $\F_q^*$, the group $G(q)$
contains the permutation of\/ $\F_q^*$ induced by $wx$, which is a $(q-1)$-cycle.
\end{lemma}

\begin{proof} Let $\sigma$ and $\rho$ be the elements of $G(q)$ induced by $f(x)$ and $wf(x)$, respectively.
Then $\rho\sigma^{-1}$ is induced by $wx$, which is a $(q-1)$-cycle.
\end{proof}

\begin{cor} \label{cyclelemma} Let $q>5$ be a prime power which cannot be written as $2^p$
with $p$ prime.  Then, for any generator $w$ of\/ $\F_q^*$, the group $G(q)$ contains the map induced by $wx$,
which is a $(q-1)$-cycle.
\end{cor}

\begin{proof}
In light of Lemma~\ref{cyclelemma0}, it suffices to show that there is a permutation binomial over $\F_q$ in
which both terms have degrees between $1$ and $q-1$.
If $q=r^e$ with $e>1$ and $r>2$, this follows from Proposition~\ref{prop1}.
If $q$ is odd and $q>5$ then it follows from Corollary~\ref{prop3cor} with $s=2$.
\end{proof}

\begin{remark} We note that $G(q)$ is the trivial group if $q=2^p$ where $2^p-1$ is prime.
For, in this case any $0<m<n<q$ will satisfy $\gcd(n-m,q-1)=1$, so that every element
of $\F_q^*$ has a unique $(n-m)$-th root in $\F_q^*$.  For any $a,b\in\F_q^*$
it follows that $ax^m+bx^n$ does not permute $\F_q$, since it takes value $0$ when $x$ is either $0$ or the
$(n-m)$-th root of $-a/b$ in $\F_q^*$.  Thus, if there are infinitely many Mersenne primes then there are
infinitely many prime powers $q$ for which $G(q)=1$.
\end{remark}

\begin{remark} We know very little about $G(q)$ when $q=2^p$ where $p$ is prime but $2^p-1$ is composite.
Proposition~\ref{prop1} implies that $G(q)$ is nontrivial if $q-1$ has a nontrivial divisor which is very small
compared to $q-1$ (for instance, such a divisor must be smaller than a constant times $\log q$).  But we know
nothing about $G(q)$ when $q-1$ has no such divisor: it is conceivable that $G(q)$ is always trivial in this case,
and it is also conceivable that $G(q)$ always equals $S_{q-1}$ in this case.
\end{remark}

In case $G(q)$ contains a $(q-1)$-cycle, there are only a few possibilities for a $G(q)$-invariant
partition of $\F_q^*$:

\begin{cor} \label{cyclecor} Let $q>5$ be a prime power which is either odd or a power of $4$.
Then every partition of\/ $\F_q^*$ which is preserved by $G(q)$ must consist of all the cosets of a
subgroup of\/ $\F_q^*$.
\end{cor}

\begin{proof} Let $\mathcal{P}$ be a partition of $\F_q^*$ which is preserved by $G(q)$, and let $S$
be the part in $\mathcal{P}$ which contains $1$.  For any $u\in \F_q^*$, Corollary~\ref{cyclelemma} implies
that $ux$ is in $G(q)$, so that $\mathcal{P}$ is preserved by $ux$, whence $uS$ is a part of $\mathcal{P}$.
In particular, if $v\in S$ then $vS$ is a part of $\mathcal{P}$ which contains $v$, so $vS\cap S$ is nonempty, whence
$vS=S$.  Thus
$S$ is a nonempty subset of $\F_q^*$ which is closed under multiplication, so it is a subgroup.
Finally, $\mathcal{P}$ consists of the sets $uS$ with $u\in\F_q^*$, namely the cosets of $S$ in $\F_q^*$.
\end{proof}

In what follows, if $d$ is a divisor of $q-1$ then we write $\mu_d$ for the group of $d$-th roots
of unity in $\F_q^*$.  The next result is the key tool we will use to show in certain cases that $G(q)$
does not preserve any of the nontrivial partitions of $\F_q^*$ described in Corollary~\ref{cyclecor}.

\begin{prop} \label{keyprop}
Let $q$ be a prime power, and let $d$ and $k$ be positive divisors of $q-1$ such that
$d\nmid k$.  Then there are at most $d$ elements $a\in\F_q^*$ for which $x^{k+1}+ax$
maps all elements of $\mu_d$ into the same coset of\/ $\F_q^*$ mod $\mu_d$.
\end{prop}

\begin{proof}
Fix an element $c\in\mu_d\setminus\mu_k$.
Pick $a\in\F_q^*$ for which $f(x):=x^{k+1}+ax$ maps $\mu_d$ into a coset of $\F_q^*/\mu_d$.
Since $1$ and $c$ are in $\mu_d$, there exists $b\in\mu_d$ for which
$f(c)=b\cdot f(1)$.  Thus
\[
c^{k+1} + ac = f(c) = b\cdot f(1) = b\cdot\bigl( 1 + a\bigr),
\]
so that
\[
(c-b)a=b-c^{k+1}.
\]
It follows that $c\ne b$, since otherwise the left side would be zero so also the right side would
be zero, whence $c=b=c^{k+1}$ so $c^k=1$, contradicting our hypothesis that $c\notin\mu_k$.
Thus we obtain
\[
a = \frac{b-c^{k+1}}{c-b},
\]
so in particular the value of $a$ is uniquely determined by the value of $b$ (since $c$ is fixed).
Since $b\in\mu_d$, this means there are at most $d$ choices for $a$.
\end{proof}

\begin{thm} \label{prim}
Let $r\ge 10485731$ be a prime power with $r\equiv 2\pmod{3}$ and $r\equiv \pm 3\pmod{8}$, and let $q=r^2$.
Then the group $G(q)$ is a primitive subgroup of $S_{q-1}$.
\end{thm}

\begin{proof}
By Corollary~\ref{cyclecor}, it suffices to prove that if $d$ is a proper divisor of $q-1$ such that
$G(q)$ preserves the set of cosets of $\F_q^*$ mod $\mu_d$, then $d=1$.  Let $d$ be such a divisor
of $q-1$.  By Proposition~\ref{prop1}, there are $q-1-(r+1)$ elements $a\in\F_q^*$ for which
$x^r-ax$ permutes $\F_q$.  Since each such polynomial $x^r-ax$ defines a function from $\F_q^*/\mu_d$ into itself, by Proposition~\ref{keyprop} with $k=r-1$ we conclude that either
$d$ divides $r-1$ or $d\ge q-1-(r+1)$.  But $q-1-(r+1)>(q-1)/2\ge d$, so in fact $d\mid (r-1)$.

Next, Proposition~\ref{prop2} implies that there are $2(r-1)$ elements $a\in\F_q^*$ for which
$x^{r+2}+ax$ permutes $\F_q$.  By Proposition~\ref{keyprop} with $k=r+1$, we conclude that
$d$ divides $r+1$, so $d\mid\gcd(r-1,r+1)=2$.

Finally, by Corollary~\ref{prop3cor}, there are at least three elements $a\in\F_q^*$
for which $x(x^{(q-1)/8}+a)$ permutes $\F_q$.  Our hypothesis $r\equiv \pm 3\pmod{8}$
implies that $r-1$ and $r+1$ are congruent to $2$ and $4$ mod $8$ (in some order), so
that $r^2-1\equiv 8\pmod{16}$.  By Proposition~\ref{keyprop} with $k=(q-1)/8$, we conclude
that $d$ divides $(q-1)/8$; since $(q-1)/8$ is odd and $d\mid 2$, it follows that $d=1$.
As noted above, by Corollary~\ref{cyclecor} this implies that $G(q)$ is primitive.
\end{proof}

\section{Proof of the main result}

We now prove Theorem~\ref{main}.  Let $r$ be a prime power, and write $q=r^2$.
By Corollary~\ref{cyclelemma} and Theorem~\ref{prim}, if $r$ is sufficiently large and $r\equiv 2\pmod{3}$
and $r\equiv \pm 3\pmod{8}$, then $G(q)$ is a primitive subgroup of $S_{q-1}$ which contains
a $(q-1)$-cycle.  In light of Corollary~\ref{gpcor}, it follows that $G(q)=S_{q-1}$ if $q-1$ cannot be written as $(\ell^d-1)/(\ell-1)$ with $d\ge 2$
and $\ell$ a prime power.  Let us add the requirements that $r\equiv \pm 3\pmod{7}$
and $r\equiv\pm 6\pmod{17}$.  Then $r^2-2$ is divisible by both $7$ and $17$, and hence is not
a prime power; thus $r^2-1\ne (\ell^2-1)/(\ell-1)$ for any prime power $\ell$.   Since
$r^2-1$ is even, if $r^2-1=(\ell^d-1)/(\ell-1)$ then $\ell^d=1+(r^2-1)(\ell-1)$ is odd and thus
$\ell$ is odd, whence $r^2-1= \ell^{d-1}+\ell^{d-2}+\dots+1\equiv d\pmod{2}$ implies that
$d$ is even.  Since $d>2$, we must have $d\ge 4$.

The Prime Number Theorem for arithmetic progressions implies that, as $N\to\infty$ the number of primes $r\le\sqrt{N}$
such that
\begin{itemize}
\item $r\equiv 2\pmod{3}$
\item $r\equiv\pm3\pmod{8}$
\item $r\equiv \pm 3\pmod{7}$
\item $r\equiv\pm 6\pmod{17}$
\end{itemize}
is asymptotic to 
\[
\frac{2^3}{\phi(3\cdot8\cdot7\cdot17)}\frac{\sqrt{N}}{\log(\sqrt{N})}=\frac{1}{48}\frac{\sqrt{N}}{\log(N)}.
\]
For any fixed $d\ge 4$, the number
\[
\#\Bigl\{ \text{prime powers }\ell:\, \frac{\ell^d-1}{\ell-1} \le N\Bigr\}
\]
is at most the number of prime powers $\ell$ such that $\ell^{d-1}\le N$.  In particular, there only exist
any such $\ell$ if $d\le 1+\log_2(N)$.  Summing over all $d$, we find that
\[
\#\Bigl\{ (\ell,d):\,  \ell \text{ is a prime power, } d\ge 4, \text{ and }\frac{\ell^d-1}{\ell-1} \le N\Bigr\}
\]
is at most
\[
\sum_{d=4}^{1+\lfloor \log_2(N)\rfloor} \!\!\!\sum_{\substack{\ell \text{ is a prime power} \\ \ell^{d-1}\le N}} \!\!\!\!\!\!\! 1\,\,
\le \sum_{d=4}^{1+\lfloor \log_2(N)\rfloor} \!\!\!\sum_{\substack{\ell \text{ is a prime power} \\ \ell^3\le N}} 
\!\!\!\!\!\!\! 1\,\,
\le \,\,\log_2(N)\cdot N^{1/3}.
\]
Since the ratio
\[
\frac{\log_2(N)\cdot N^{1/3}}{\sqrt{N}/(48\log(N))}
\]
approaches zero as $N\to\infty$, it follows that the number of primes $r\le \sqrt{N}$ such that
\begin{itemize}
\item $r\equiv 2\pmod{3}$
\item $r\equiv\pm 3\pmod{8}$
\item $r\equiv \pm 3\pmod{7}$
\item $r\equiv\pm 6\pmod{17}$
\item $r^2-1$ cannot be written as $(\ell^d-1)/(\ell-1)$ with $d\ge 2$
and $\ell$ a prime power
\end{itemize}
is asymptotic to $\sqrt{N}/(48\log(N))$.
By Corollary~\ref{cyclelemma} and Theorem~\ref{prim}, for any such $r$ the group $G(r^2)$ is primitive
and contains an $(r^2-1)$-cycle, and hence (by Corollary~\ref{gpcor}) equals the symmetric
group on $\F_{r^2}^*$.  This proves Theorem~\ref{main}.


\section{Prime fields}

In this section we discuss whether there are infinitely many primes $p$ for which $G(p)$ equals $S_{p-1}$.  We will focus on the question whether there are infinitely many primes $p$ for which $G(p)$ is primitive.
According to the heuristic in \cite[Section~4]{MZ}, for all sufficiently large primes $p$ we expect that every permutation binomial $ax^m+bx^n$ over $\F_p$ (with $a,b\in\F_p^*$ and $0<m<n<p$) will satisfy $\gcd(n-m,p-1) > p/(2\log p)$.  In \cite{MZ} we noted that we had verified this conclusion for all primes $p<10^5$; an independent verification
for $p<15000$ is announced in \cite{VR}.
We now show that the factor `2' in this bound plays a crucial role in connection with $G(p)$, in the sense that if this factor could be improved to a constant less than $1$ then there would only be finitely many primes $p$ for which $G(p)$ is primitive.

\begin{prop} Fix a real number $c>1$.  For any prime power $q$ which is sufficiently large compared to $c$, if $G(q)$ is primitive then there exist $a,b\in\F_q^*$ and $0<m<n<q$ such that $ax^m+bx^n$ induces a permutation of\/ $\F_q$ and $\gcd(n-m,q-1)<c(q-1)/\log q$.
\end{prop}

\begin{proof}
Let $q$ be a prime power such that $G(q)$ is primitive but all permutation binomials $ax^m+bx^n$
over $\F_q$ have $\gcd(n-m,q-1)\ge c(q-1)/\log q$.  Primitivity implies in particular that there is no divisor $k$ of $q-1$ such that $1<k<q-1$ and $G(q)$ induces a permutation on $\F_q^*/\mu_k$, where $\mu_k$ denotes the group of $k$-th roots of unity in $\F_q^*$.  It follows that the gcd of all the numbers $\gcd(n-m,q-1)$ for which $ax^m+bx^n$ permutes $\F_q$ (with $a,b\in\F_q^*$ and $0<m<n<q$)
must be $1$.  Writing $\gcd(n-m,q-1)=(q-1)/d$ where $d\mid (q-1)$, it follows that
\[
1 = \gcd(\{(q-1)/d:\, d\mid (q-1) \text{ and } d\le (\log q)/c\}),
\]
or equivalently
\[
q-1 = \lcm(\{d:\, d\mid (q-1) \text{ and } d\le (\log q)/c\}),
\]
which can be rewritten as
\[
q-1 = \lcm(\{d:\, d\mid (q-1),\, \text{ $d$ is a prime power,\, and } d\le (\log q)/c\}),
\]
or equivalently
\[
q-1\,\,\,\,\,\le\!\!\!\!\!\!\!\!\!\!\!\! \prod_{\substack{d\mid (q-1) \\ d\le (\log q)/c \\ d=p^k \text{ with $p$ prime and $k\ge 1$}}} \!\!\!\!\!\!\!\!\!\!\!\!\!\! p.
\]
Removing the condition $d\mid (q-1)$ can only increase the right side; if we remove this condition and then take logs of both sides, we obtain
\[
\log (q-1) \,\,\,\,\,\le\!\!\!\!\!\!\!\!\!\!\!\! \sum_{\substack{d\le (\log q)/c \\ d=p^k \text{ with $p$ prime and $k\ge 1$}}} \!\!\!\!\!\!\!\!\!\!\!\!\!\!\log p.
\]
By the Prime Number Theorem, the right side is asymptotic to $(\log q)/c$ as $q\to\infty$, so for sufficiently large $q$ the right side is smaller than the left side.  This contradiction completes the proof.
\end{proof}

\begin{remark}\label{sound}
Correspondence with Igor Shparlinski and Kannan Sound\-ararajan yielded a heuristic argument suggesting a converse
to the above result.  Namely, suppose there exists a number $c<1$ such that, if $q$ is sufficiently large $q$ and
$0<m<n<q$ satisfy
$\gcd(m,n,q-1)=1$ and $\gcd(n-m,q-1)>cq/\log q$, then there exist $a,b\in\F_q^*$ such that $ax^m+bx^n$
permutes $\F_q$.  We do not know whether such a number $c$ should exist, but a result in this direction
(with $c$ replaced by $2\log\log q$) is proved in \cite[Thm.~3.1]{MZ}.  Our heuristic suggests that, if such a number $c<1$ exists, then there should be infinitely many primes $q$ for which $G(q)$ is primitive.
\end{remark}


\section{Concluding remarks}

We have shown that $G(q)$ equals $S_{q-1}$ for many $q$'s which are squares of primes: in fact, for a density-$1$ subset of those $q$'s which are squares of the primes in certain arithmetic progressions.  We do not know whether $G(q)$ equals $S_{q-1}$ for a density-$1$ subset of the $q$'s which are squares of primes.  We also do not know how often $G(q)$ equals $S_{q-1}$ for other types of prime powers $q$.  In particular, does this happen for infinitely many primes $q$?  We suspect that it happens whenever $q$ is a sufficiently large power of $4$.

When $G(q)$ does not equal $S_{q-1}$, it would be interesting to investigate
 what the group $G(q)$ turns out to be.  
Let $r(q)$ be the greatest common divisor of all numbers of the form $\gcd(n-m,q-1)$ where $0<m<n<q$
and there exist $a,b\in\F_q^*$ such that $ax^m+bx^n$ permutes $\F_q$.
Proposition~\ref{keyprop}, and even moreso its proof, suggests that usually $r(q)$ will be the largest proper divisor $d$ of $q-1$
for which $G(q)$ permutes the cosets of $\F_q^*$ mod $\mu_d$.  When this happens, one might guess that $G(q)$ usually equals
the full group of permutations of $\F_q$ induced by polynomials of the form $x^i h(x^{r(q)})$ with $i>0$.  The latter group was determined by Wan and Lidl \cite{WL}: it is the semidirect
product of $(\Z/r(q)\Z)^*$ by the wreath product
$(\Z/r(q)\Z)\wr S_{(q-1)/r(q)}$.  It seems that one can at least show that $G(q)$ contains a copy of $S_{(q-1)/r(q)}$ under some hypotheses, since the action of $G(q)$ on $\F_q^*/\mu_{r(q)}$ induces a map $G(q)\to S_{(q-1)/r(q)}$ whose image is primitive and
contains a $(q-1)/r(q)$-cycle.  It would be interesting to investigate this further.


\end{document}